\documentclass{amsart}
 \usepackage[margin=1.5cm]{geometry}
 \usepackage[utf8x]{inputenc} 
\numberwithin{equation}{section}
\usepackage{amssymb}
\usepackage{amsmath, amsfonts,amsthm,amssymb,amscd, verbatim,graphicx,color,multirow,booktabs, caption,tikz,tikz-cd, mathdots,bm}
\usepackage{tikz-cd}
 \usepackage{hyperref}
\usetikzlibrary{positioning}

\newtheorem{theorem}{Theorem}
\newtheorem{lemma}[theorem]{Lemma}

\newtheorem{proposition}[theorem]{Proposition}
 \newtheorem{corollary}[theorem]{Corollary}

      \theoremstyle{definition}
       \newtheorem*{definition}{Definition}
     \newtheorem{example}[theorem]{Example}
     
     \theoremstyle{remark}
     \newtheorem{remark}[theorem]{Remark}

\newcommand{\Sym}{\mathop{\mathrm{Sym}}}
\newcommand{\Alt}{\mathop{\mathrm{Alt}}}

 \definecolor{mycolor}{rgb}{0.55,0.0,0.16}
  \definecolor{myred}{rgb}{0.6,0.0,0.16}
  \definecolor{mygreen}{rgb}{0.0,0.6,0.16}
  \definecolor{myviolet}{rgb}{1,0,1}

\hypersetup{
colorlinks=false,
allbordercolors=myred,
citebordercolor=mygreen
}

\makeatletter
\@namedef{subjclassname@2020}{%
  \textup{2020} Mathematics Subject Classification}
\makeatother

\begin{document}
 \title[The growth of abelian sections]{The growth of abelian sections}
\author[L. Sabatini]{Luca Sabatini}

\address{Luca Sabatini, Alfr\'ed R\'enyi Institute of Mathematics, Hungarian Academy of Sciences\newline
Re\'altanoda utca 13-15, H-1053, Budapest, Hungary} 

\email{sabatini@renyi.hu, sabatini.math@gmail.com}
	\maketitle

        \begin{abstract}
        Given an abstract group $G$, we study the function $ab_n(G) := \sup_{|G:H| \leq n} |H/[H,H]|$.
        If $G$ has no abelian composition factors, then $ab_n(G)$ is bounded by a polynomial:
        as a consequence, we find a sharp upper bound for the representation growth of these groups.
          \end{abstract}
          
          \vspace{0.2cm}
\section{Introduction}

	 Let $G$ be an abstract group.
	 There are many functions in group theory that describe $G$ starting ``from the top'',
	 the most famous probably being the subgroup growth \cite{2002LS},
	 which counts the number of subgroups of any given index.
	 The main goal of this article is to study the abelian quotients of these subgroups,
	 i.e. the function
 $$ ab_n(G) \> := \> \sup_{|G:H| \leq n} |H/H'| \> , $$   
     where $H':=[H,H]$ is the commutator subgroup of $H$.
     We call it the {\bfseries abelianization growth} of $G$.
     As we will see, this is really a matter of finite groups,
     and it is quite related to the representation theory of $G$.
	 Let us define
   $Rep_n(G)$ as the number of (pairwise inequivalent) irreducible complex $G$-representations of degree {\itshape at most} $n$ and finite image.
     We will refer to this function as the {\bfseries representation growth} of $G$.
     We remark that the usual definition of $Rep_n(G)$
     involves representations with infinite image as well:
     for technical reasons, we will consider only finite images.
Representation growth was introduced in the early 2000s,
and was especially studied for arithmetic groups (we refer the reader to the survey \cite{2013Klopsch}).
On the other hand, the quantitative study of abelian sections has emerged recently.
Perhaps the first question one can ask about $ab_n(G)$ is whether it really {\itshape grows}.
With some effort made in Section \ref{sect2},
results from \cite{2021SabA,2021SabB} give the following tight bound
(all the logarithms in this paper are to base $2$):

 \begin{theorem} \label{thLBAbG}
     If $G$ is any residually finite infinite group, then
     $$ ab_{n}(G) \> \geq \> n^{\> 1/32 \log \log n} $$
     for infinitely many $n$.\\
     On the other hand, there exists a residually finite infinite group $G$ such that
     $$ ab_{n}(G) \> \leq \> n^{\> \beta/\log \log n} $$
     for some fixed $\beta>0$ and every $n \geq 3$.
     \end{theorem}
     \vspace{0.1cm}
     
Sections \ref{sect3} and \ref{sect4} concern both $ab_n(G)$ and $Rep_n(G)$.
As it is a common observation that $ab_1(G) = Rep_1(G)$ when $G$ is finite,
 it happens that $ab_n(G)$ and $Rep_n(G)$ are still related at larger scales.
  Indeed, \cite[Lemma 2.6 (a)]{2004LM} (where $ab_n(G)$ appears in passing) provides
  \begin{equation} \label{eqLM}
  ab_n(G) \> \leq \> n \cdot Rep_n(G) 
\end{equation}   
  for every finite group $G$ and $n \geq 1$. On the other hand, we show

 \begin{proposition} \label{propRGBAb} 
     There exists an absolute constant $C>0$ such that if $G$ is $d$-generated for some $d \geq 1$, then
     $$ Rep_n(G) \> \leq \> (n! \cdot ab_{(n+1)!}(G))^{ \> Cd \> \cdot \> (\log ab_{(n+1)!}(G) \> + \> n \log n)} $$
     for every $n \geq 1$.
     \end{proposition}
     \vspace{0.1cm}
     
     Moreover, we construct residually finite infinite groups with arbitrarily fast
     abelianization and representation growths (Proposition \ref{propFastAG}).
     We move to study more in detail the groups without abelian composition factors.
     
     \begin{theorem} \label{thPAGoSG}
     There exists an absolute constant $\alpha>0$ such that the following holds.
    Let $C \geq 1$, $D \geq 0$. Let $G$ be a group with $D$ abelian composition factors,
    and suppose that all of them have order at most $C$.
   Then
   $$ ab_n(G) \> \leq \> C^D \cdot n^\alpha $$
   for every $n \geq 1$.
     \end{theorem}
 \vspace{0.1cm}
 
 This result depends on the classification of finite simple groups.
 One construction of Kassabov and Nikolov \cite{2006KN} provides a finitely generated residually finite group
 without abelian composition factors and representation growth of factorial type (see Example \ref{exKN}).
Thus $ab_n(G)$ and $Rep_n(G)$ can be quite different in general.
In the range of $d$-generated cartesian products without abelian composition factors,
they define {\itshape admissible} \cite[Theorem 1.8]{2006KN} a representation growth
which is bounded below a factorial.
We prove that the same bound holds for all groups without abelian composition factors.

\begin{theorem} \label{thKN}
There exists an absolute constant $\alpha'>0$ such that,
if $G$ is a $d$-generated group without abelian composition factors, then
$$ Rep_n(G) \> \leq \> (n!)^{\alpha' d} $$
for every $n \geq 1$.
\end{theorem}
\vspace{0.1cm}

The proof uses $ab_n(G) \leq n^\alpha$,
the Jordan theorem on finite subgroups of $GL_n(\mathbb{C})$,
and that the number of boundedly generated groups of order $n$ without abelian composition factors is polynomially bounded \cite{2011JZP}.
In Section \ref{sect5} we generalize $ab_n(G)$ and $Rep_n(G)$ to transitive group actions,
in place of the regular action of $G$ on itself:
while (\ref{eqLM}) resists in the generalized setting,
the same is not true for any inequality in the direction of Proposition \ref{propRGBAb}.

\vspace{0.2cm}
 \begin{section}{Basic properties of $ab_n(G)$}  \label{sect2}
     
     By definition we have $ab_1(G)=|G/G'|$. We accept infinite values for $ab_n(G)$,
   	and if $G$ is finite, then $ab_{|G|}(G)$ is the size of the largest abelian section of $G$.
    One might wonder whether ``$\sup$'' is often unnecessary in the definition of $ab_n(G)$, and ``$\max$'' is acceptable.
    A satisfactory answer comes from the starting point of the theory of subgroup growth,
    i.e. the fact that a finitely generated group has finitely many subgroups of index $n$,
    for every fixed $n \geq 1$.
    More precisely, the following result,
    where $Sub_n(G)$ denotes the number of subgroups of index {\itshape at most} $n$ in $G$,
    can be deduced from \cite[Corollary 3.4]{2002LS}.
    
     \begin{lemma} \label{lemSubGrowth}
    Let $G$ be a $d$-generated group, for some $d \geq 1$.
    Then $Sub_n(G)$ is finite for every $n \geq 1$, and
    $$ Sub_n(G) \> \leq \> (n!)^{d} $$
    for all $n \geq 4$.
    Moreover, if $G$ is solvable, then
    $$ Sub_n(G) \> \leq \> 4^{d n} $$
    is true for all sufficiently large $n$.
    \end{lemma}
    \vspace{0.1cm}
     
     We develop the theory of abelianization growth in a general context,
     and use the hypothesis of finite generation only when really needed.
     
     \begin{remark} \label{remHZentrum}
     If $Z(G)$ denotes the center of $G$, then $(HZ(G))'=H'$ for every $H \leqslant G$.
     Thus $ab_n(G)$ is determined by some subgroup which contains the center, for every $n \geq 1$.
     \end{remark}
     \vspace{0.1cm}
      
      The following are easily checked (slightly more general results will be proved in Section \ref{sect5}).
      
       \begin{lemma}[Hereditary properties] \label{lemAbGSub}
     Let $H \leqslant G$ and $n \geq 1$. We have
     \begin{itemize}
     \item[(i)] $ab_n(H) \leq ab_{|G:H|n}(G)$;
        \item[(ii)] $ab_n(G) \leq |G:H| \cdot ab_n(H)$.
     \end{itemize}
     Moreover, if $N \lhd G$, then $ab_n(G/N) \leq ab_n(G)$ for all $n \geq 1$.
     \end{lemma}
     \vspace{0.1cm}

     \begin{subsection}{Residually finite and profinite groups} \label{subsec2.1}
     
   	Let $\mathcal{G}$ be a family of finite groups. We define
   	$$ ab_n(\mathcal{G}) \> := \> \sup_{G \in \mathcal{G}} \hspace{0.1cm} ab_n(G) \> . $$ 
     It is obvious that $ab_n(\mathcal{G}) \leq ab_n(\mathcal{F})$ for all $n \geq 1$ if 
     $\mathcal{G}$ and $\mathcal{F}$ are two families of finite groups and $\mathcal{G} \subseteq \mathcal{F}$.
     We recall a definition, and introduce its finitary version.
     
     \begin{definition}
     A (typically infinite) group $G$ has {\bfseries FAb} ({\bfseries Finite Abelianizations})
      if $|H/H'|$ is finite for every subgroup $H \leqslant G$ of finite index.\\
     A family of finite groups $\mathcal{G}$ has {\bfseries BAb} ({\bfseries Bounded Abelianizations})
     if $ab_n(\mathcal{G})$ is finite for all $n \geq 1$.
     \end{definition}
     \vspace{0.1cm}

     We stress that the definition of FAb is slightly weaker than the property of having $ab_n(G)<\infty$ for every $n \geq 1$.
     Indeed, the two are equivalent when $G$ is finitely generated.
     The following is useful to construct groups with FAb.
     
      \begin{lemma} \label{lemPerfFAb}
     If $(G_k)_{k \geq 1}$ is a sequence of finite perfect groups,
     then the cartesian product $G := \prod_{k=1}^{\infty} G_k$ has FAb.
     \end{lemma}
     \begin{proof}
    Let $H \leqslant G$ be a subgroup of finite index.
    Let $I_H$ be the set of indices $k$ such that
    $H$ intersects the direct factor $G_k$ properly.
    Of course, $H$ is contained in $\prod_{k \notin I_H} G_k \times \prod_{k \in I_H} M_k$, where $M_k<G_k$ are maximal subgroups.
    Since $H$ has finite index, $I_H$ is a finite set.
    Then, we use the fact that the $G_k$'s are perfect to write $|H/H'| \leq \prod_{k \in I_H} |M_k| < \infty$.
     \end{proof}
     \vspace{0.1cm}
     
     The following group has FAb, but $ab_{60}(G)=\infty$.
     
     \begin{example} \label{exEx}
     Let $p \geq 2$ be a prime, let $\mathbb{F}_p$ be the field with $p$ elements,
     and consider the permutational wreath product $(\mathbb{F}_p)^5 \rtimes \Alt(5)$.
     The abelianization of this group has size $p$, and it is unbounded.
     We overcome this problem by considering the subspace of $(\mathbb{F}_p)^5$
     which is constituted by the vectors whose coordinates sum to zero.
     This is a $4$-dimensional space, the so-called {\itshape deleted permutation module}, and let us denote it by $V_p$.
     We restrict the semidirect product to this subspace, and define
      $$ G_p \> := \> V_p \rtimes \Alt(5) \>, $$
      where the action is the one described above.
       Let $G:=\prod_{p \mbox{\tiny{ prime}}} G_p$ be the cartesian product of the $G_p$'s.
     Since $G_p$ is perfect for every $p$, $G$ has FAb from Lemma \ref{lemPerfFAb}.
     On the other hand, for every $p$, it is easy to find a subgroup of $G$ of index $|\Alt(5)|=60$,
     whose abelianization is precisely $V_p$.
     \end{example}
     \vspace{0.1cm}
     
     A stronger condition assures that $ab_n(G)$ is finite for every $n \geq 1$.
     
      \begin{lemma} \label{lemPerfBAb}
     Let $(G_k)_{k \geq 1}$ be a sequence of finite perfect groups with BAb.
   If $G:=\prod_{k=1}^{\infty} G_k$, then $ab_n(G)$ is finite for every $n \geq 1$.
     \end{lemma}
     \begin{proof}
     Let $\mathcal{G}:=(G_k)_{k \geq 1}$, and let $H \leqslant G$ be a subgroup of finite index.
    Arguing as in the proof of Lemma \ref{lemPerfFAb}, we see that $I_H$ is a finite set.
    Moreover $2^{|I_H|} \leq \prod_{k \in I_H} |G_k:M_k| \leq |G:H|$, and so $|I_H| \leq \log |G:H|$.
    Since the $G_k$'s are perfect, for every $n \geq 1$ we have
     \begin{align*}
ab_n(G)  & \> = \> 
\sup_{|G:H| \leq n} |H/H'| \\ & \> \leq \> 
 ( \sup_{G_k \in \mathcal{G}} ab_n(G_k) )^{|I_H|}  \\ & \> \leq \> 
( ab_n(\mathcal{G}) )^{\log n} \> , 
\end{align*}
and the conclusion follows because $\mathcal{G}$ has BAb.
     \end{proof}
     \vspace{0.1cm}
     
   	Now we see how $ab_n(G)$ is related to the finite quotients of $G$.
     
     \begin{proposition}[Reduction to finite groups] \label{propFAbBAb}
     Let $G$ be an arbitrary group,
     and let $\mathcal{G}$ be the family of its finite quotients.
     Then 
     $$ ab_n(\mathcal{G}) \> \leq \> ab_n(G) $$
     for all $n \geq 1$.
     Moreover, equality holds in the following cases:
     
     \begin{itemize}
       \item[(i)] $G$ has FAb;
     \item[(ii)] $G$ is finitely generated and residually finite
     (so in this case $\mathcal{G}$ has BAb if and only if $G$ has FAb).
     \end{itemize}
     
     \end{proposition}
     \begin{proof}
      From the last part of Lemma \ref{lemAbGSub}, for every $n \geq 1$,
     $$ ab_n(\mathcal{G}) = \sup_{G/N \in \mathcal{G}} ab_n(G/N) \leq \sup_{G/N \in \mathcal{G}} ab_n(G) = ab_n(G) . $$
      Now let $G$ have FAb, and let $H \leqslant G$ be a subgroup of finite index. It follows that $|G:H'|=|G:H||H:H'|$ is finite.
     If $N := \cap_{x \in G} (H')^x \leqslant H'$ is the normal core of $H'$ in $G$, we have $|G:N|<\infty$ and
     $$ |H/H'| = |(H/N):(H/N)'| , \hspace{0.5cm} \mbox{and} \hspace{0.5cm} |G:H| = |(G/N):(H/N)| . $$
     Since $H/N \leqslant G/N \in \mathcal{G}$, we obtain $ab_n(G) \leq ab_n(\mathcal{G})$.\\
     Finally, let $G$ be finitely generated and residually finite.
     If $G$ has FAb, then the conclusion follows from (i), so suppose that $G$ does not have FAb.
     Since $G$ is finitely generated, there are only finitely many subgroups for any given index.
     Then there exists a minimal integer, say $m$,
   	 such that $G$ has a subgroup of index $m$ with infinite abelianization.
     Hence $ab_{m-1}(G) <\infty$, while $ab_m(G)=\infty$. For all $n<m$, we have
     $ab_n(\mathcal{G}) = ab_n(G)$ by the same argument we used above.
     It remains to prove that $ab_m(\mathcal{G})=\infty$.
     Let $H \leqslant G$ be a subgroup of index $m$ and infinite abelianization.
     Since $H/H'$ is infinite and residually finite, there exists a subgroup $J/H' \leqslant H/H'$ of finite but arbitrarily large index $|H:J|$.
     Passing to the normal core $N:=\cap_{x \in G} J^x$,
     there exists a normal subgroup $N \lhd G$ inside $J$.
    
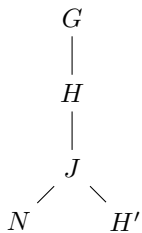
  \begin{figure}[h]
\centering
\begin{tikzpicture}[node distance=1cm]

\node(G)                           {$G$};
\node(H)       [below of=G] {$H$};
\node(J)      [below of=H]  {$J$};
\node(der) [below left of=J] {$N$};
\node(N) [below right of=J] {$H'$};

\draw (G) -- (H);
\draw (H) -- (J);
\draw (J) -- (der);
\draw (J) -- (N);

\end{tikzpicture}
\end{figure}
     
     Consider $G/N \in \mathcal{G}$, and $H/N \leqslant G/N$ as a subgroup of index $m$. We have
     $$ |(H/N):(H/N)'| = |H:NH'| \geq |H:J| . $$
     Since $|H:J|$ is arbitrarily large, the result follows.
     \end{proof}
     \vspace{0.1cm}
     
   Three remarks about Proposition \ref{propFAbBAb}:
     
     \begin{itemize}
     \item the hypotheses in (i) and (ii) are necessary.
     If $G \cong \mathbb{Q}$, the additive group of the rationals,
     then $\mathcal{G}=\{ 1 \}$, while $ab_n(G)=\infty$ for all $n \geq 1$;
     \item essentially the same proof works if $G$ is the inverse limit of the inverse system $\mathcal{G}$;
   	\item if $G$ satisfies one among (i)-(ii), then one can take the profinite completion
   	$\widetilde{G}$ and conclude that $ab_n(G)=ab_n(\widetilde{G})$ for all $n \geq 1$.
     However, we stress that the approach with families is more general,
     because not all the families of finite groups arise as
     finite quotients of an infinite mother group.
     As we will see with Proposition \ref{propCraven}, not all the techniques of profinite groups work for families.
     \end{itemize}
     \vspace{0.1cm}
     
     For the proof of Theorem \ref{thLBAbG}, we will use a number of known results.
     We start with the following two.
     
     \begin{theorem}[Theorem 1 in \cite{2021SabA}] \label{thProc}
    Every finite group $G$ of size at least $3$ contains an abelian section of size at least $|G|^{1/32 \log\log|G|}$.
     \end{theorem}
     
     \begin{theorem}[Theorem 1 in \cite{2021SabB}] \label{thArkMath}
     Let $\mathcal{G}:=(\Sym(k))_{k \geq 5}$.
     Then there exists an absolute constant $\alpha>0$ such that
     $$ ab_n(\mathcal{G}) \> \leq \> n^{\> \alpha/\log\log n} $$
     for every $n \geq 3$.
     \end{theorem}
     \vspace{0.1cm}
     
     From Lemma \ref{lemAbGSub} (i), it is easy to see that Theorem \ref{thArkMath} remains true with $\Alt(k)$ in place of $\Sym(k)$.
     The next result is very useful, especially when combined with Proposition \ref{propFAbBAb}.
     
     \begin{theorem}[Theorem 1.2 in \cite{2006KN}] \label{thKN}
     If a cartesian product of alternating groups $\prod_{k=5}^{\infty} \Alt(k)^{f(k)}$ is topologically finitely generated,
     then it is the profinite completion of a finitely generated residually finite group.
     \end{theorem}
     \vspace{0.1cm}
     
     We also need a numerical claim.
     
     \begin{lemma} \label{lemNum}
     If $x$ and $y$ are real numbers and $8 \leq x \leq y$, then
     $$ \frac{x}{\log x} + \frac{ y}{2 \log y} \> \leq \> \frac{x+y}{\log(x+y)} \> . $$
     \end{lemma}
     \begin{proof}
     We have equality when $x=y=8$, and it is easy to check the case $x=y \geq 8$.
     Moreover, for every fixed $x$, the function
     $f(y) := \frac{x+y}{\log(x+y)} - \frac{ y}{2 \log y}$ is not-decreasing.
     \end{proof}
     \vspace{0.1cm}
    
      \begin{proof}[Proof of Theorem \ref{thLBAbG}]
      The lower bound for $ab_n(G)$ follows from Lemma \ref{lemAbGSub} and Theorem \ref{thProc}.
      In fact, for every $N \lhd G$ of finite index one has
      $$ ab_{|G/N|}(G) \>\geq \> ab_{|G/N|}(G/N) \> \geq \> |G/N|^{1/32\log\log|G/N|} \> . $$
      If $G$ is residually finite, then there exist normal subgroups of finite but arbitrarily large index.\\
      To prove the second part, we construct a residually finite infinite group which attains the bound of the symmetric groups
      (this is not obvious, because there is no infinite group whose finite quotients are precisely the symmetric groups).
      Let $(a_j)_{j \geq 1}$ be a sequence of positive integers to be fixed later, and let $G := \prod_{j=1}^{\infty} \Alt(a_j)$.
      We remark immediately that $G$ has FAb from Lemma \ref{lemPerfFAb}.
      We will show that $ab_n(G) \leq n^{\beta/\log\log n}$ for all $n \geq 3$.
      From Theorem \ref{thKN},
      $G$ is the profinite completion of a finitely generated residually finite group,
      and so the proof would follow from Proposition \ref{propFAbBAb}.
      To estimate $ab_n(G)$, we work with its finite quotients.
      Indeed $G$ is the inverse limit of $(G_k)_{k \geq 1}$, where $G_k := \prod_{j=1}^k \Alt(a_j)$ for all $k \geq 1$.
      From Proposition \ref{propFAbBAb} (i), it is enough to prove that for every $k \geq 0$ and $H \leqslant G_{k+1}$ one has
      \begin{equation} \label{eqTh}
       \log|H/H'| \> \leq \> \beta \cdot \frac{\log|G_{k+1}:H|}{\log\log|G_{k+1}:H|} \> , 
	\end{equation}      
      where $\beta:=2 \alpha$, and $\alpha>0$ is the constant in Theorem \ref{thArkMath}.
      We accomplish this task by induction on $k$.
      Arguing as in \cite[Sect. 4]{2021SabB}, we can suppose
     $H=A \times B$ for some $A \leqslant G_k$ and $B \leqslant \Alt(a_{k+1})$.
     By the inductive hypothesis we have
     \begin{align*}
\log|H/H'| & \> = \> 
\log|A/A'| + \log|B/B'| \\ & \> \leq \> 
\beta \frac{\log|G_k:A|}{\log\log|G_k:A|} + \alpha \frac{\log(|\Alt(a_{k+1})|/|B|)}{\log\log(|\Alt(a_{k+1})|/|B|)} \> .
\end{align*}
 Since $(\Alt(a_j))_{j \geq 1}$ has BAb, $(G_k)_{k \geq 1}$ has BAb from Lemma \ref{lemPerfBAb} and Proposition \ref{propFAbBAb}.
     Therefore, we can assume that $|G_k:A|$ is at least $2^8$
     (otherwise the first term is bounded by an absolute constant).
Moreover, since $\Alt(a_{k+1})$ is perfect, we can assume $|\Alt(a_{k+1})|/|B| \geq a_{k+1}$.
Let $x:=\log|G_k:A|$ and $y:=\log(|\Alt(a_{k+1})|/|B|)$.
      Now, if $(a_j)_{j \geq 1}$ is sufficiently fast, then the condition $8 \leq x \leq y$ is satisfied,
      and so (\ref{eqTh}) follows from Lemma \ref{lemNum}.
      \end{proof}
      \vspace{0.1cm}
      
       Here is a characterization of BAb.
       Although we will not use it, it can be interesting in itself.
     
     \begin{lemma} \label{lemEquivBAb}
     Let $\mathcal{G}$ be a family of finite groups. The following are equivalent:
     
     \begin{itemize}
     \item[(i)] for every {\bfseries normal} subgroup $N \lhd G \in \mathcal{G}$, $|N/N'|$ can be bounded just from $|G:N|$;
     \item[(ii)] $\mathcal{G}$ has BAb;
     \item[(iii)] every solvable section $H/N$ of every $G \in \mathcal{G}$ has size bounded just from $|G:H|$ and its derived length.
     \end{itemize}
     
     \end{lemma}
     \begin{proof}
      ``(iii) $\Rightarrow$ (i)'' is obvious.\\
 ``(i) $\Rightarrow$ (ii)''  Let $H \leqslant G \in \mathcal{G}$, and let $N := \cap_{x \in G} H^x$ be its normal core.
We have
\begin{align*}
|H/H'| & \> = \> 
|H:NH'| |NH':H'|  \\ & \> \leq \> 
|G:N| |N:N \cap H'|  \\ & \> \leq \>
|G:N| |N:N'| \> ,
\end{align*}
and the last term can be bounded from $|G:N|$ by hypothesis.
   Since $G/N$ is a permutation group over $|G:H|$ elements, we have $|G/N| \leq |G:H|!$.\\
 ``(ii) $\Rightarrow$ (iii)'' We work by induction on the derived length $\ell$.
 The case $\ell=1$ is the definition of BAb, so suppose that $H/N$ is a solvable section of $G \in \mathcal{G}$,
 and let $\ell+1$ be its derived length.
 We have $|H/N|=|H:NH'||NH':N|$, and $|H:NH'|$ can be bounded from $|G:H|$, because $|H:NH'| \leq |H:H'|$.
 Moreover,
 $NH'/N \cong (H/N)'$ is a solvable section of derived length $\ell$,
 and so by induction $|NH':N|$ can be bounded from $|G:NH'|$, and
 $$ |G:NH'| = |G:H||H:NH'| \leq |G:H||H:H'| . $$
 Since this can be bounded from $|G:H|$ by hypothesis, the proof follows.
     \end{proof}
     \vspace{0.1cm}
     
     The previous proof shows that the abelianization growth can be (crudely) estimated
     by a check on abelian quotients of normal subgroups.
     In particular, we remark that $|G/\cap_{x \in G} H^x| \leq 3^{|G:H|}$
     if $G/\cap_{x \in G} H^x$ is solvable \cite{1967Dixon}.
      
     \end{subsection}
     
     \end{section}

     \vspace{0.2cm}
     \begin{section}{Abelianization and representation growths}  \label{sect3}
     
      The counting of representation growth usually concerns a fixed infinite group $G$.
     The corresponding notion for a family of finite groups $\mathcal{G}$ is \index{$Rep_n(\mathcal{G})$|textbf}
    $$ Rep_n(\mathcal{G}) \> := \> \sup_{G \in \mathcal{G}} \hspace{0.1cm} Rep_n(G) \> . $$ 
  If representations with infinite image are considered,
    the matter of representation growth cannot be reduced to the matter of finite groups,
    not even for a finitely generated residually finite group.
    With our restricted definition we have
    
    \begin{lemma} \label{propReductionRG}
   Let $G$ be an arbitrary group, and let $\mathcal{G}$ be the family of its finite quotients.
   Then
     $$ Rep_n(G) \> = \> Rep_n(\mathcal{G}) $$
     for every $n \geq 1$.
    \end{lemma}
    \begin{proof}
    For every $N \lhd G$ of finite index, a representation of $G/N$ can be seen as a representation of $G$ with finite image.
   It follows that
    $$ Rep_n(\mathcal{G}) = \sup_{G/N \in \mathcal{G}} Rep_n(G/N) \leq Rep_n(G) . $$
    To prove equality,
    we have to show that two different $G$-representations with finite image are already different in some finite quotient of $G$.
   Let $\pi, \rho \in Irr_n(G)$ for some $n \geq 1$, and let $\pi \neq \rho$.
   Since both $Ker(\pi)$ and $Ker(\rho)$ have finite index in $G$, then $N:=Ker(\pi) \cap Ker(\rho)$ has finite index,
  and $\pi$ and $\rho$ are distinct representations of $G/N$.
    \end{proof}
    \vspace{0.1cm}
     
     Essentially, the next is taken from \cite{2006KN}.
     
      \begin{example} \label{exKN}
      For every $d \geq 3$ and $k \geq 5$, the direct product $(\Alt(k))^{(k!/2)^{d-2}}$ is $d$-generated \cite[Lemma 2]{1978Wie}.
      Thus, from Theorem \ref{thKN} and the discussion below \cite[Theorem 1.4]{2006KN} in that paper,
      there exists a $d'$-generated residually finite group $G$ with profinite completion
      $\prod_{k=5}^{\infty} (\Alt(k))^{(k!/2)^{d-2}}$, where $d':=22(d+1)$.
      For every $k \geq 5$, $G$ has at least $(k!/2)^{d-2}$ different representations of degree $k-1$,
      one for each factor in the corresponding direct product. It follows that
      $$ Rep_n(G) \> \geq \> ((n+1)!/2)^{d-2} \> \geq \> (n!)^{d-2} \geq (n!)^{d'/50} $$
      for every $n \geq 4$.
      \end{example}
      \vspace{0.1cm}
     
     Now we recall two results in representation theory:
     the first is the famous Jordan theorem about the finite subgroups of $GL_n(\mathbb{C})$.
     
     \begin{theorem}[Jordan function] \label{thJSC}
     There exists a function $j: \mathbb{N}_+ \rightarrow \mathbb{N}_+$ such that the following hold:
     \begin{itemize}
     \item[(i)] if $G$ is a finite subgroup of $GL_n(\mathbb{C})$,
     then $G$ has an abelian normal subgroup of index
     at most $j(n)$ \cite{1878Jordan}; 
     \item[(ii)] $j(n) \leq (n+1)!$ for every $n \geq 71$ \cite[Theorem A]{2007Collins}.
     \end{itemize}
     \end{theorem}
     \vspace{0.1cm}
     
     The second result is the representation growth version of Lemma \ref{lemAbGSub},
     and it is a clever application of the Frobenius reciprocity law.
     We give our proof, because it will be useful in Section \ref{sect5}.
  
     \begin{lemma}[Lemma 2.2 in \cite{2004LM}] \label{lemRepGSub}
     Let $G$ be a finite group, $H \leqslant G$ and $n \geq 1$. We have
     
     \begin{itemize}
     \item[(i)] $Rep_n(H) \leq |G:H| \cdot Rep_{n|G:H|}(G)$;
     \item[(ii)] $Rep_n(G) \leq |G:H| \cdot Rep_n(H)$.
     \end{itemize}
     
     \end{lemma}
     \begin{proof}
     For every $n \geq 1$, let us denote by $Irr_n(G)$ the set of the irreducible $G$-representations (up to isomorphism)
     of degree at most $n$.
     We define a map $\phi: Irr_n(H) \rightarrow Irr_{n|G:H|}(G)$ in the following way:
     for every $\rho \in Irr_n(H)$, choose $\phi(\rho)$ as an irreducible constituent of the induced representation $\rho^{\uparrow G}$.
    Whenever $\phi(\rho)$ is a constituent of $\pi^{\uparrow G}$ for some $\pi \in Irr_n(H)$,
     by Frobenius reciprocity $\pi$ is a constituent of the restriction $\phi(\rho)_{\downarrow H}$.
     Fix $\rho \in Irr_n(H)$, and let
     $\pi$ be of minimal degree among the irreducible representations in $\phi(\rho)_{\downarrow H}$.
     We have
     \begin{align*}
|\phi^{-1}\phi(\rho)| \dim(\pi)  &  \> \leq \>
\dim(\phi(\rho)_{\downarrow H}) \\ & \>  = \>
\dim(\phi(\rho)) \\ &  \> \leq \>
 \dim(\pi^{\uparrow G}) \\ & \>  = \>
 |G:H| \dim(\pi) \> .
\end{align*}
     Thus the fibers of $\phi$ have size at most $|G:H|$, and (i) follows.\\
     For (ii), similarly define a map $\varphi: Irr_n(G) \rightarrow Irr_n(H)$ in the following way:
     for every $\rho \in Irr_n(G)$, choose $\varphi(\rho)$ as an irreducible constituent of $\rho_{\downarrow H}$.
     Whenever $\varphi(\rho)$ is a constituent of $\pi_{\downarrow H}$ for some $\pi \in Irr_n(G)$,
     by Frobenius reciprocity $\pi$ is a constituent of $\varphi(\rho)^{\uparrow G}$.
     Fix $\rho \in Irr_n(G)$,
     and let $\pi$ be of minimal degree among the irreducible representations in $\varphi(\rho)^{\uparrow G}$
     We have
       \begin{align*}
|\varphi^{-1}\varphi(\rho)| \dim(\pi)  & \> \leq \>
\dim(\varphi(\rho)^{\uparrow G}) \\ & \> = \>
|G:H| \dim(\varphi(\rho))  \\ & \> \leq \>
 |G:H| \dim(\pi_{\downarrow H}) \\ & \> = \>
 |G:H| \dim(\pi) \> ,
\end{align*}
	and the proof follows as before.
	We remark that (\ref{eqLM}) is obtained by setting $n=1$ in (i).
     \end{proof}
     \vspace{0.1cm}
     
    \begin{proof}[Proof of Proposition \ref{propRGBAb}]
    Fix $n \geq 1$.
    From Theorem \ref{thJSC} (i),
    for every representation $\rho : G \rightarrow GL_n(\mathbb{C})$,
    the finite image $\rho(G)$ has an abelian subgroup of index at most $j(n)$.
    Let $A_\rho \leqslant G$ be the pre-image of such an abelian subgroup.
    By the correspondence theorem we have $|G:A_\rho| = |\rho(G):\rho(A_\rho)| \leq j(n)$.
    Moreover, by the abelianity of $\rho(A_\rho)$,
    $$ |\rho(A_\rho)| \leq |A_\rho/(A_\rho)'| \leq ab_{|G:A_\rho|}(G) \leq ab_{j(n)}(G) , $$
	and then
	$$ |\rho(G)| =  |\rho(G):\rho(A_\rho)| \cdot |\rho(A_\rho)| \leq j(n) \cdot ab_{j(n)}(G) . $$ 
	From the main result of Lubotzky \cite{2001Lub}, and Remark 1 in that paper,
	there exist at most $|\rho(G)|^{3d \cdot \log|\rho(G)|}$ abstract groups of order at least $|\rho(G)|$.
	For each of these abstract groups, there are at most $|\rho(G)|$ irreducible representations up to equivalence,
	the exact number being the number of conjugacy classes.
	Then, there are at most $|\rho(G)|$ possible embeddings into $GL_n(\mathbb{C})$ up to conjugation.
	The number of homomorphisms from $G$ to one fixed embedding is at most $|\rho(G)|^d$.
	To sum up
	$$ Rep_n(G) \leq |\rho(G)|^{3d \log|\rho(G)|+1+d} \leq 
	(j(n) \cdot ab_{j(n)}(G))^{ 4d \log( j(n) \cdot ab_{j(n)}(G))} . $$
	When $n \leq 70$, the proof follows by choosing the constant $C$ in the statement sufficiently large.
	For $n \geq 71$, we use Theorem \ref{thJSC} (ii) and the trivial inequality $(n+1)! \leq (n!)^2$.
	 \end{proof}
	 \vspace{0.1cm}
	
    We can improve Proposition \ref{propRGBAb} in some cases.
    A finite group is called {\bfseries monomial}
    if each of its irreducible representations is induced from some $1$-dimensional representation of some subgroup.
    It is known that the class of monomial groups lies between those of nilpotent and solvable groups \cite[Corollary 5.13]{1976Isaacs}
    (originally attributed to Taketa).
     
       \begin{lemma} \label{lemRGMono}
     If $G$ is a $d$-generated monomial group, then
       $$ Rep_n(G) \> \leq \> 4^{d n} \cdot ab_n(G) $$
       for all sufficiently large $n$.
     \end{lemma}
     \begin{proof}
   For all $n \geq 1$, every representation of degree at most $n$ is induced by a $1$-dimensional representation of some subgroup
   $H \leqslant G$ of index at most $n$.
   This implies that
   $$ Rep_n(G) \leq Sub_n(G) \cdot Rep_{_1}(H) \leq Sub_n(G) \cdot ab_n(G) . $$
   As we said above, monomial groups are solvable,
   and so the proof follows from Lemma \ref{lemSubGrowth}.
     \end{proof}
   \vspace{0.1cm}
     
     Combining (\ref{eqLM}) and Lemma \ref{lemRGMono} we obtain
     
     \begin{corollary} \label{corMono}
     Let $\mathcal{G}$ be a family of $d$-generated monomial groups, for some fixed $d \geq 1$.
     Then $ab_n(\mathcal{G})$ is at most exponential if and only if $Rep_n(\mathcal{G})$ is at most exponential.
     \end{corollary}
     \vspace{0.1cm}

     \begin{subsection}{Upper and lower bounds} \label{subsec3.1}
     
     In general, there is no upper bound for $ab_n(G)$ and $Rep_n(G)$.
     To see this, we modify the construction of Example \ref{exEx}.
     
     \begin{proposition} \label{propFastAG}
     For every function $f: \mathbb{N}_+ \rightarrow \mathbb{N}_+$ there exists a
     finitely generated FAb residually finite group $G$ such that $ab_n(G) \geq f(n)$ for infinitely many $n$.
     The same is true for $Rep_n(G)$.
     \end{proposition}
     \begin{proof}
     The statement for $Rep_n(G)$ follows from that for $ab_n(G)$ and (\ref{eqLM}).
     We divide the proof for $ab_n(G)$ in three steps:
     \vspace{0.1cm}
     
     {\itshape Step 1.}
     Let $(p_k)_{k \geq 5}$ be a sequence of primes to be fixed later,
     and consider the permutational wreath product $(\mathbb{F}_{p_k})^k \rtimes \Alt(k)$.
    Let $V_k$ be the $(k-1)$-dimensional subspace of $(\mathbb{F}_{p_k})^k$
     which is constituted by the vectors whose coordinates sum to zero.
     We restrict the semidirect product to this subspace, and define
      $$ G_k \> := \> V_k \rtimes \Alt(k) $$
       for all $k \geq 5$.
       Let $G:=\prod_{k \geq 5} G_k$ be the cartesian product of the $G_k$'s.
       \vspace{0.1cm}
       
     {\itshape Step 2.}
     Since $G_k$ is perfect for every $k \geq 5$, $G$ has FAb from Lemma \ref{lemPerfFAb}
     (actually we could use Lemma \ref{lemPerfBAb}, but this is not required). 
     On the other hand, let $H := G_5 \times ... \times G_{k-1} \times V_k \leqslant G_k$.
     We have $|G:H|=k!/2$, and $|H/H'| = (p_k)^{k-1}$.
     Hence, for a sufficiently fast sequence of primes $(p_k)_{k \geq 5}$,
     we obtain $ab_{k!/2}(G) \geq f(k!/2)$ for every $k \geq 5$.
     \vspace{0.1cm}
     
     {\itshape Step 3.}
     By Theorem \ref{thKN}, $\prod_{k \geq 5} \Alt(k)$ is the profinite completion
     of a finitely generated residually finite group.
     If $v \in V_k$ is a non-trivial element, then $G_k$ is generated by $\Alt(k)$ and $[v,\Alt(k)]$,
     and so \cite[Lemma 2.4]{2006KN} gives that $G$ itself is the profinite completion of a finitely generated residually finite group.
     The proof is completed by Proposition \ref{propFAbBAb}.
     \end{proof}
     \vspace{0.1cm}
     
     We conclude this section with some remarks about lower bounds.
     It follows from the definitions of $ab_n(\mathcal{G})$ and $Rep_n(\mathcal{G})$ that
      \begin{equation} \label{eqLBT}
      ab_{|G|}(\mathcal{G}) \> \geq \> ab_{|G|}(G) \hspace{0.8cm} 
      \mbox{ and } \hspace{0.8cm}  Rep_{|G|}(\mathcal{G}) \> \geq \> Rep_{|G|}(G) 
	\end{equation}    
	 for every family of finite groups $\mathcal{G}$ and every $G \in \mathcal{G}$.
	 Of course, $ab_{|G|}(G)$ is the size of the largest abelian section of $G$.
	 From a standard result in representation theory, $Rep_{|G|}(G)$ equals the number of conjugacy classes of $G$, which we denote by $k(G)$.
	The inequalities in (\ref{eqLBT}) are not effective if $ab_n(\mathcal{G})$ or $Rep_n(\mathcal{G})$
	 are fast (in the extreme case, if $\mathcal{G}$ does not have BAb).
      However, we saw in the proof of Theorem \ref{thLBAbG} that the first inequality is surprisingly precise in giving a general lower bound
      (by ``general'' here we mean that it is true for every $\mathcal{G}$).
      Can one do the same for $Rep_n(\mathcal{G})$?
     The best known lower bound for $k(G)$ is due to the companion Baumeister-Maroti-Tong-Viet \cite{2016BMV},
     and it is the following:
     for every $\varepsilon >0$ there exists $C_\varepsilon>0$ such that
     \begin{equation} \label{eqBMV}
     k(G) \> \geq \> C_\varepsilon \cdot \frac{\log |G|}{(\log \log |G|)^{3+\varepsilon}} 
	\end{equation}     
     for every finite group $G$ of size at least $3$. On the other side,
     there are arbitrarily large groups satisfying $k(G) \leq C \cdot (\log |G| / \log\log|G|)^2$ (see \cite[pag. 375]{1997Pyber}).
     When combined to (\ref{eqLBT}), (\ref{eqBMV}) yields the best lower bound we can provide at the moment,
     for the representation growth of every family of arbitrarily large finite groups.
     It is worth noticing that Craven \cite{2010Craven} used a different idea to obtain a better bound,
     which holds only for a fixed infinite profinite group $G$, but not for an arbitrary family of finite groups.
     
    \begin{proposition}[Proposition 2.2 in \cite{2010Craven}] \label{propCraven}
    For every $\varepsilon>0$ there exists $C_\varepsilon>0$ such that the following holds.
     If $G$ is an infinite, finitely generated profinite group with FAb, then
     $$ Rep_n(G) \> \geq \> C_\varepsilon \cdot (\log n)(\log \log n)^{1-\varepsilon} $$
     for infinitely many $n$.
     \end{proposition}
     \vspace{0.1cm}
     
     In short, Craven's proof distinguishes if $G$ has finitely many or infinitely many maximal subgroups.
     In the first case $G$ is virtually pro-nilpotent,
     while in the second case it maps onto arbitrarily large finite groups with trivial Frattini subgroup.
     In both cases, the representation growth is as fast as claimed.
     As the second part of Theorem \ref{thLBAbG} points out, the same approach fails for $ab_n(G)$,
     essentially because of the relatively slow abelianization growth of alternating groups.
     
     \end{subsection}
     
     \end{section}

     \vspace{0.2cm}
     \begin{section}{Polynomial abelianization growth} \label{sect4}
     
     \begin{definition}
     Let $G$ be a group, and let $\mathcal{G}$ be a family of finite groups.
     We say that $G$ ($\mathcal{G}$, respectively) has {\bfseries Polynomial Abelianization Growth}
     if there exist absolute constants $C,\alpha >0$ such that
     $$ ab_n(G) \> \leq \>  C n^\alpha $$
     for every $n \geq 1$ ($ab_n(\mathcal{G}) \leq C n^\alpha$, respectively).
     The notion of {\bfseries Polynomial Representation Growth} is defined similarly.
     \end{definition}
     \vspace{0.1cm}
    
     We remark that (\ref{eqLM}) says that PRG implies PAG in every residually finite group.
     Here are some cases where PAG is realized.
   
   \begin{lemma} \label{lemQuasi}
   Let $G$ be a finite group, and $\varepsilon>0$.
   The following hold:
   
   \begin{itemize}
   \item[(i)]  if $|G/G'| \leq \varepsilon^{-1}$,
     and $|G:H| \geq |G|^\varepsilon$ for every proper subgroup $H < G$,
     then
  	$$ ab_n(G) \> \leq \> \varepsilon^{-1} \cdot n^{\varepsilon^{-1} -1} $$
   for every $n \geq 1$;
   \item[(ii)] if $\dim(\pi) \geq |G|^\varepsilon$ for every non-trivial irreducible representation $\pi$ of $G$,
   then 
   $$ Rep_n(G) \> \leq \> n^{\varepsilon^{-1} -1} $$
   for every $n \geq 1$.
   \end{itemize}
   
   \end{lemma}
   \begin{proof}
   (i) For every proper subgroup $H<G$ we have
     $$ |H| \leq |G|^{1-\varepsilon} = |G:H|^{1-\varepsilon}|H|^{1-\varepsilon} , $$
     and arranging the terms we obtain $|H|^\varepsilon \leq |G:H|^{1-\varepsilon}$, so that
     $$ |H/H'| \leq |H| \leq |G:H|^{\varepsilon^{-1}-1} . $$
     Combined with $|G/G'| \leq \varepsilon^{-1}$, this gives the desired inequality.\\
     (ii)
      By hypothesis, we have $Rep_n(G) = 1$ for every $n < |G|^\varepsilon$.
   Let $n \geq |G|^\varepsilon$, and let $r_n(G)$ be the number of (inequivalent) irreducible representations of degree {\itshape exactly} $n$.
  Since $|G| = \sum_{\pi \in Irr(G)} (\dim \pi)^2$, we have
  $$ r_n(G) \leq |G|/n^2 \leq n^{\varepsilon^{-1} -2} . $$
  Finally
  \begin{align*}
Rep_n(G)  & \> = \>
1 + \sum_{j=\lceil |G|^\varepsilon \rceil}^n r_j(G) \\ & \> \leq \>
1 + (n-\lceil |G|^\varepsilon \rceil +1) n^{\varepsilon^{-1} -2}  \\ & \> \leq \>
 n^{\varepsilon^{-1} -1} \> . \qedhere 
\end{align*}
   \end{proof}
   \vspace{0.1cm}
   
   Sometimes the condition $\dim(\pi) \geq |G|^\varepsilon$ for every non-trivial $\pi$ and fixed $\varepsilon>0$
   is called $\varepsilon$-{\bfseries quasirandomness}.
   It is a property shared by quasisimple groups of Lie type of bounded Lie rank \cite{1974LS},
   and it implies what is asked in Lemma \ref{lemQuasi} (i).
   On the other hand,
   Theorem \ref{thPAGoSG} and Example \ref{exKN} show that PAG does not imply PRG in general, not even for finitely generated groups.
   We are going to prove Theorem \ref{thPAGoSG}.
   When $H$ is an arbitrary group and $N \lhd H$, one has
   \begin{align*}
|H/H'| & \> = \> 
|H:NH'| \cdot |NH': H'| \\ & \> = \> 
|(H/N):(H/N)'| \cdot |N : N \cap H'| \> ,
\end{align*}
   and so $|H/H'| \leq |(H/N):(H/N)'||N:N'|$, because $N' \leqslant N \cap H'$.
   
    \begin{lemma}[Product formula] \label{lemAbGExt}
    Let $N \lhd G$. Then, for every $n \geq 1$,
    $$ ab_n (G) \> \leq \> \max_{j \> | \> n} \> ( ab_{n/j} (G/N) \cdot ab_j (N) ) \> . $$
     In particular,
   $$ ab_n (G/N) \> \leq \> ab_n (G) \> \leq \> ab_n (G/N) \cdot ab_n (N) $$
    for all $n \geq 1$.
    \end{lemma}
    \begin{proof}
    Let $n \geq 1$ and $H \leqslant G$ with $|G:H| \leq n$.
    From the previous inequality we have
    \begin{align*}
|H/H'| & \> \leq \> 
|(H/N \cap H):(H/N \cap H)'| \cdot |N \cap H : (N \cap H)'| \\ & \> = \> 
|(NH/N):(NH/N)'| \cdot |N \cap H : (N \cap H)'| \> .
\end{align*}
    Moreover,
    $$ |(G/N):(NH/N)| = |G:NH|, \hspace{1cm} \mbox{and} \hspace{1cm} |N: N \cap H| = |NH:H| . $$
  	Let $j:=|NH:H|$. The three computations above combined give
  	$$ |H/H'| \leq ab_{|G:H|/j}(G/N) \cdot ab_{j}(N) . $$
  	The proof follows.
    \end{proof}
     \vspace{0.1cm}

     We divide the proof of Theorem \ref{thPAGoSG} in three steps:
     \vspace{0.1cm}
     
     {\bfseries \MakeUppercase{\romannumeral 1.} Reduction to $G$ finite}.
     This is required to use the classification of non-abelian finite simple groups.
     Let $G$ be a group with $D$ abelian composition factors, which size is at most $C$.
     Let $H \leqslant G$ be a subgroup of finite index and let $N:=\cap_{x \in G} H^x$ be its normal core.
     Arguing as in the proof of Lemma \ref{lemEquivBAb} we have
     \begin{align*}
|H/H'| & \> = \> 
|H:NH'| |NH':H'|  \\ & \> \leq \> 
|G:N| |N:N \cap H'|  \\ & \> \leq \>
|G:H|! \cdot |N:N'| \> .
\end{align*}
By the Jordan-H\"older theorem we have $|N/N'| \leq C^D$.
Thus $G$ has FAb, and the reduction follows from Proposition \ref{propFAbBAb} (i).
     \vspace{0.1cm}
     
     {\bfseries \MakeUppercase{\romannumeral 2.} One simple group}.
   	Let $S$ be a non-abelian finite simple group.
     This step follows from the work of Liebeck and Shalev \cite{2005LS} on character degrees of finite simple groups:
     they proved the following theorem about the zeta function
     $$ \zeta_S(t) \> := \> \sum_{\pi \in Irr(S)} (\dim \pi)^{-t} \> , $$
     where $\pi$ ranges among the (inequivalent) irreducible representations of $S$.
     
     \begin{theorem}[Theorem 1.1 in \cite{2005LS}]
     Let $S$ be a non-abelian finite simple group. If $t>1$, then
     $$ \zeta_S(t) \rightarrow 1 \hspace{0.5cm} \mbox{ as } \hspace{0.5cm} |S| \rightarrow +\infty . $$
     \end{theorem}
     \vspace{0.1cm}
     
     We remark that a weaker fact is needed to prove Theorem \ref{thPAGoSG},
     namely that there exists $t_0$ such that $\zeta_S(t_0)$ is bounded when $|S| \rightarrow +\infty$.
     Now, it is easy to see that the zeta function is a sum involving representation growth.
     Let $r_n(S)$ be the number of (inequivalent) irreducible representations of degree {\itshape exactly} $n$.
     Arranging the terms we can write
     $$ \zeta_S(2) \> = \> \sum_{\pi \in Irr(S)} (\dim \pi)^{-2} \> = \>  \sum_{n \geq 1} r_n(S) n^{-2} \> . $$
     Since the series converges when $|S|$ grows,
     there exists an absolute constant $c>0$ such that $r_n(S) \leq c n^2$ for all $n$ and $S$.
     Hence
    $$ Rep_n(S) \> = \> \sum_{j =1}^n r_j(S) \> \leq \> c n^3 \> , $$
     and combining this inequality with (\ref{eqLM}), we obtain
     $$ ab_n (S) \> \leq \> n \cdot Rep_n(S) \> \leq \> cn^4 \> . $$ 
     
     We point out that it would be very interesting to have a CFSG-free proof of the inequality $|H/H'| \leq |G:H|^\alpha$
     in non-abelian finite simple groups.
      \vspace{0.1cm}
     
        {\bfseries \MakeUppercase{\romannumeral 3.} Mixing simple groups}.
     Since $S$ is perfect, we can stretch $cn^4$ to a monic polynomial $n^\alpha$. In fact, for $n = 2$ we have
     $$ cn^4 = n^{4 + \log(c)} , $$
     while ``$<$'' holds for all larger $n$.
     Thus we can write $ab_n(S) \leq n^\alpha$ for some fixed $\alpha>0$ and every $n \geq 1$.
     Now let $G$ be a finite group with $D$ abelian composition factors of size at most $C$.
     We work by induction on the composition length, so let $N \lhd G$ be a maximal normal subgroup.
     If the simple quotient $G/N$ is non-abelian,
     then $ab_n(G/N) \leq n^\alpha$ by the previous step, and $ab_n(N) \leq C^D n^{\alpha}$ by induction.
     Using Lemma \ref{lemAbGExt}, we obtain
     $$ ab_n(G) \leq \max_{j \> | \> n} \> ((n/j)^\alpha C^D (j)^\alpha) = C^D n^\alpha $$
     for all $n \geq 1$.
     When some abelian quotient $G/N$ appears, then by induction $ab_n(N) \leq C^{D-1} n^{\alpha}$.
     Let $H \leqslant G$.
     Since $|NH/N| \leq |G/N| \leq C$, we can write
     \begin{align*}
|H/H'|  & \> \leq \> 
|(H/N \cap H):(H/N \cap H)'| \cdot |N \cap H : (N \cap H)'| \\ & \> \leq \> 
 |NH/N| \cdot ab_{_{|N:N \cap H|}}(N) \\ & \> \leq \> 
 C \cdot C^{D-1} |N:N \cap H|^\alpha \\ & \> \leq \> 
 C^D |G:H|^\alpha .
\end{align*}
     This concludes the proof of Theorem \ref{thPAGoSG}.
     \vspace{0.1cm}
     
     \begin{remark}
      It is not true that mixing (via group extensions) infinitely many groups with PAG,
     we always do obtain a group with PAG. Actually, we do if only finitely many of the starting groups are not perfect.
     Moreover, (\ref{eqLM}) and Proposition \ref{propRGBAb} imply that,
     for a family $\mathcal{G}$ of finite groups which are all $d$-generated for a fixed $d$,
     $ab_n(\mathcal{G})$ is finite for every $n \geq 1$ if and only if the same is true for $Rep_n(\mathcal{G})$.
      Let $G^k := G \times ... \times G$ be the direct product of $k$ copies of a perfect group $G$,
      and let $\mathcal{G} := (G^k)_{k \geq 1}$.
      Then $\mathcal{G}$ has BAb, but certainly $Rep_{|G|}(\mathcal{G})$ is not finite.
    This shows that the dependence on $d$ in Proposition \ref{propRGBAb} is necessary.
     \end{remark}
     \vspace{0.1cm}
     
      Characterizing groups with PAG (or even PRG) seems to be a serious task:
      
      \begin{example}[$p$-groups with PRG]
      Highlighting a nice connection,
      Lubotzky and Martin \cite{2004LM} proved that an arithmetic group has polynomial representation growth
      if and only if it has the so-called Congruence Subgroup Property
      (informally, this means that all the finite index subgroups arise from the arithmetic structure of the group).
      This property is own, for example, by $SL_r(\mathbb{Z})$ when $r \geq 3$
      (notice that $SL_2(\mathbb{Z})$ does not even have FAb, because of its free subgroups of finite index).
      Let $G:=SL_r(\mathbb{Z})$, and for fixed $p \geq 2$ and every $k \geq 1$ define
      $$ G_k \> := \> Ker \{ \> SL_r(\mathbb{Z}) \twoheadrightarrow SL_r(\mathbb{Z}/p^k \mathbb{Z}) \> \} \> . $$
      Now $G_1$ has finite index in $G$, and so it has polynomial representation growth from Lemma \ref{lemRepGSub} (i).
      Moreover, it is not hard to check that $G_1/G_k$ is a $p$-group for every $k \geq 1$.
      It follows (from Lemma \ref{propReductionRG}, for example)
      that $(G_1/G_k)_{k \geq 1}$ is a family of $p$-groups with polynomial representation growth.
      \end{example}
       \vspace{0.1cm}
       
       To prove Theorem \ref{thKN},
       we need a sharp bound on the number of $d$-generated groups without abelian composition factors.
       We quote \cite[Corollary 13.4]{2011JZP}:
       
       \begin{theorem}[Jaikin-Zapirain and Pyber] \label{thJP}
        There exists a constant $C$ such that the number of $d$-generated finite groups of order
        $n$ without abelian composition factors is at most $n^{Cd}$.
       \end{theorem}
       \vspace{0.1cm}
       
       It is easy to see that, up to replacing $C$, the same bound is true for the number of $d$-generated finite groups
       of order at most $n$ without abelian composition factors.
    
      \begin{proof}[Proof of Theorem \ref{thKN}]
      Arguing as in the proof of Proposition \ref{propRGBAb}, we obtain an inequality of the type
	$$ Rep_n(G) \> \leq \> \phi_d ( \> \> n! \cdot ab_{(n+1)!}(G) \> \> )  \cdot (n! \cdot ab_{(n+1)!}(G))^{d+1} \> , $$
	where
	$\phi_d(n)$ is the number of $d$-generated groups of order at most $n$ and without abelian composition factors.
	From Theorem \ref{thPAGoSG} we have $ab_n(G) \leq n^\alpha$ for some fixed $\alpha>0$.
	Moreover, Theorem \ref{thJP} gives $\phi_d(n) \leq n^{Cd}$ for some absolute constant $C>0$.
	Up to replacing with a larger $\alpha'(\alpha,C)$, the inequality
	$$ Rep_n(G) \leq ( n! \cdot ((n+1)!)^\alpha)^{Cd+d+1} \leq (n!)^{\alpha' d} $$
	is true for every $n \geq 1$.
      \end{proof}
      
     \end{section}

     \vspace{0.2cm}
\section{Group actions} \label{sect5}

Let the group $G$ act transitively on the set $\Omega$.
We count the irreducible representations which appear in the permutation representation,
and measure the abelian sections of $G$ above a stabilizer.

  \begin{definition} \index{$ab_n(G,Y)$, $Rep_n(G,Y)$|textbf} 
 	   Let $G$ be a group, and $Y \leqslant G$. For every $n \geq 1$, we define
 	    $$ ab_n(G,Y) \> := \> \sup_{Y \leqslant H \leqslant G \atop |G:H| \leq n} |H/H'Y| \> , $$ 
 	    where $H$ ranges among the subgroups between $Y$ and $G$
 	    (notice that $H'Y \lhd H$).
 	   Moreover, we define $Rep_n(G,Y)$ as the number of (inequivalent) irreducible $G$-representations
 	    having a non-zero $Y$-fixed vector,
 	   degree at most $n$, and finite image.
 	   \end{definition}
 	   \vspace{0.1cm}
 	    
 	   We shall use $Irr_n(G,Y)$ to denote the set of the irreducible $G$-representations (up to isomorphism)
 	  having a non-zero $Y$-fixed vector, degree at most $n$, and finite image.
 	  Here is a simple connection with the matter of the previous sections.

\begin{lemma}[Reduction to faithful actions] \label{4lemRTNS}
      If $N \leqslant Y$ and $N \lhd G$, then for every $n \geq 1$
     $$ ab_n(G,Y) \> = \> ab_n(G/N,Y/N) , \hspace{0.5cm}
     \mbox{ and }
     \hspace{0.5cm} Rep_n(G,Y) \> = \> Rep_n(G/N,Y/N) . $$ 
     \end{lemma}
     \begin{proof}
    Let us recall that the intermediate subgroups $N \leqslant H \leqslant G$
    are in bijection (via the correspondence theorem) with the subgroups of the quotient $G/N$.
    Let $N \leqslant Y \leqslant H \leqslant G$. We have
     $$ |(H/N):(H/N)'(Y/N)|  = |(H/N):(H'Y/N)| = |H/H'Y| . $$
     Moreover $|(G/N):(H/N)|=|G:H|$, so the proof of the first part follows.\\
     Now the induced representation $(1_Y)^{\uparrow G}$ is contained in $(1_N)^{\uparrow G}$,
     which is just the regular representation of $G/N$.
     So the representations in $Irr_n(G,Y)$ are representations of $G/N$.
     Moreover, $(1_Y)^{\uparrow G}$ is equivalent to $(1_{(Y/N)})^{\uparrow (G/N)}$ when seen in the quotient,
     and the proof follows.
     \end{proof}
     \vspace{0.1cm}
     
     The following is the generalized version of Proposition \ref{propFAbBAb}.
     Essentially the proof is the same, and we leave it to the reader.
     
      \begin{proposition}[Reduction to finite groups]
 	  Let $G$ be an arbitrary group, let $\mathcal{G}$ be the family of its finite quotients, and let $Y \leqslant G$. Then
 	  $$ ab_n(G,Y) \> \geq \> \sup_{G/N \in \mathcal{G}} ab_n \left( \frac{G}{N} , \frac{YN}{N} \right) $$
 	  for every $n \geq 1$.
 	  Equality holds in the following cases:
 	  \begin{itemize}
 	  \item[(i)] $ab_n(G,Y)$ is finite for all $n \geq 1$;
 	  \item[(ii)] $G$ is finitely generated and $\cap_{|G:H|<\infty} H \leqslant Y$.
	\end{itemize} 	   
     \end{proposition}
      \vspace{0.1cm}
      
      To not borden unnecessarily the discussion, we do not introduce any notation of the type
      $ab_n(\mathcal{G})$ or $Rep_n(\mathcal{G})$ in the generalized setting.
      When $G$ is residually finite and the index of $Y$ is infinite,
     a remarkable difference with the case $Y \lhd G$ is that there are no non-trivial lower bounds
     for $ab_n(G,Y)$ and $Rep_n(G,Y)$ in general.
     This is precisely because there is no such a lower bound for the largest abelian section between $Y$ and $G$,
     nor for the number of representations with a non-trivial $Y$-fixed vector.

     \begin{example}
     Let $G$ be a $2$-transitive permutation group, and let $Y \leqslant G$ be the stabilizer of a point.
    Then the permutation representation is composed by two irreducible $G$-representations,
    one being the trivial representation.
    Moreover, since $Y$ is maximal and not normal, there is no non-trivial abelian section between $Y$ and $G$.
    It follows that
    $ab_n(G,Y) \leq 1$ and $Rep_n(G,Y) \leq 2$ for all $n \geq 1$.
     \end{example}

     \begin{lemma}
     Let $G$ be residually finite, and let $|G:Y|=\infty$.
	Then $ab_n(G,Y)$ is bounded by an absolute constant if and only if
	$|N_G(H):H|$ is bounded by an absolute constant, for every $Y < H \leqslant G$ with $|G:H|<\infty$.
	In particular, $ab_n(G,Y)=1$ for all $n \geq 1$ if and only if all the subgroups of finite index containing $Y$ are self-normalizing.
     \end{lemma}
     \begin{proof}
     Let $\theta \geq 1$ such that $ab_n(G,Y) \leq \theta$ for all $n \geq 1$.
     This means that every abelian section (of finite index) between $Y$ and $G$ has size at most $\theta$.
     Now, the size of a finite group is bounded by a function on the size of the largest abelian section (see Theorem \ref{thProc}, for example).
     Then the size of {\itshape every} section of finite index between $Y$ and $G$ is bounded by a function on $\theta$. 
     But this means exactly that the set $\{|N_G(H):H| \}_{Y < H \leqslant G, \\ |G:H|<\infty}$ is bounded by a function on $\theta$.\\
     The proof of the second part follows, because every non-trivial finite group has a non-trivial abelian section.
     \end{proof}
     \vspace{0.1cm}
     
      The following are the generalized versions of Lemmas \ref{lemAbGSub} and \ref{lemRepGSub}.
     
     \begin{lemma}[Relative $ab_n$ and subgroups] \label{4lemAbGSub}
     Let $Y \leqslant H \leqslant G$, $|G:H|<\infty$, and $n \geq 1$. We have
     
     \begin{itemize}
     \item[(i)] $ab_n(H,Y) \leq ab_{|G:H|n}(G,Y)$;
        \item[(ii)] $ab_n(G,Y) \leq |G:H| \cdot ab_n(H,Y)$.
     \end{itemize}
     
     \end{lemma}
     \begin{proof}
     Choose any $Y \leqslant J \leqslant H$ with $|H:J| \leq n$.
     Since $J$ is also a subgroup of $G$, and $|G:J|=|G:H||H:J|$, (i) follows.\\
     To prove (ii), choose any $Y \leqslant J \leqslant G$ with $|G:J| \leq n$.
     Since $Y \leqslant H \cap J$, we have
     \begin{align*}
|J/J'Y| & \> = \> 
|J:J'(H \cap J)| \cdot |J'(H \cap J):J'Y| \\ & \> = \> 
|J:J'(H \cap J)| \cdot |(J'Y)(H \cap J) : J'Y| \\ & \> \leq \> 
|J:H \cap J| \cdot |H \cap J : H \cap (J'Y)| \> .
\end{align*}
     Now $(H \cap J)'Y \leqslant H \cap (J'Y)$, and so
     $$ |J/J'Y| \leq |G:H| \cdot |H \cap J : (H \cap J)'Y| . $$
     Since $|H:H \cap J| \leq |G:J|$, the proof follows.
     \end{proof}
     \vspace{0.0cm}
     
      \begin{lemma}[Relative $Rep_n$ and subgroups] \label{4lemRepGSub}
     Let $G$ be a finite group, $Y \leqslant H \leqslant G$, and $n \geq 1$. We have
     
     \begin{itemize}
     \item[(i)] $Rep_n(H,Y) \leq |G:H| \cdot Rep_{n|G:H|}(G,Y)$;
     \item[(ii)] $Rep_n(G,Y) \leq |G:H| \cdot Rep_n(H,Y)$.
     \end{itemize}
     
     \end{lemma}
     \begin{proof}
     Let us recover the proof of Lemma \ref{lemRepGSub}.
     Thus, it rests to prove that the two functions $\phi: Irr_n(H) \rightarrow Irr_{n|G:H|}(G)$ and $\varphi:Irr_n(G) \rightarrow Irr(H)$,
     as defined in that proof, map into $Irr_{n|G:H|}(G,Y)$ and $Irr_n(H,Y)$ when restricted to $Irr_n(H,Y)$ and $Irr_n(G,Y)$ respectively.\\
     For (i), let $\rho \in Irr_n(H,Y)$, and let $\phi(\rho)$ be an irreducible constituent of the induced representation $\rho^{\uparrow G}$.
     By Frobenius reciprocity, $\rho$ is a constituent of the restriction $\phi(\rho)_{\downarrow H}$,
     and so $\rho_{\downarrow Y}$ is contained in $\phi(\rho)_{\downarrow Y}$.
     But, again from Frobenius reciprocity, $\rho_{\downarrow Y}$ contains the trivial $Y$-representation.
     This implies that $\phi(\rho)_{\downarrow Y}$ contains the trivial $Y$-representation and,
     by a last application of Frobenius reciprocity,
     that $\phi(\rho)$ is a constituent of $(1_Y)^{\uparrow G}$ as desired.\\
     For (ii), let $\rho \in Irr_n(G,Y)$.
     We remark that, to construct $\varphi$,
     we can {\itshape choose} $\varphi(\rho)$ to be any irreducible constituent of $\rho_{\downarrow H}$.
     We use this fact to force $\varphi(\rho)$ to be also an irreducible constituent of $(1_Y)^{\uparrow H}$.
     This can always be done, in fact, let $\chi:=trace(\rho)$ be the character of $\rho$.
     Let us use $(1_Y)^{\uparrow H}$ and $(1_Y)^{\uparrow G}$ to denote both the permutation representations and the correspondent characters.
     If $\langle \cdot , \cdot \rangle_G$ denotes the standard inner product, by Frobenius reciprocity we have
     $$ \langle \chi_{\downarrow H} , (1_Y)^{\uparrow H} \rangle_H = \langle \chi , ((1_Y)^{\uparrow H})^{\uparrow G} \rangle_G =
      \langle \chi , (1_Y)^{\uparrow G} \rangle_G \geq 1 , $$
      where the last inequality follows because $\rho \in Irr(G,Y)$ by hypothesis.
      Thus, $\rho_{\downarrow H}$ and $(1_Y)^{\uparrow H}$ share some irreducible components, and the proof is complete.
     \end{proof}
     \vspace{0.1cm}
     
   When compared to the strong relationship between $ab_n(G)$ and $Rep_n(G)$
   (which is settled by (\ref{eqLM}) and Proposition \ref{propRGBAb}),
   the one between $ab_n(G,Y)$ and $Rep_n(G,Y)$ is quite milder.
    
    \begin{proposition} \label{4propRAbGRG1}
    Let $Y \leqslant G$. The following hold
    \begin{itemize}
    \item[(i)] $ab_1(G,Y) = Rep_1(G,Y)$;
    \item[(ii)] $ab_n(G,Y) \leq n \cdot Rep_n(G,Y)$ for every $n \geq 1$.
    \end{itemize}
    \end{proposition}
    \begin{proof}
    It follows from the definition that $ab_1(G,Y)=|G/G'Y|$.
    Now we notice that 
    $$ Rep_1(G,Y) = Rep_1 \left( \frac{G}{G'},\frac{G'Y}{G'} \right) , $$
    and since $(G'Y/G') \lhd (G/G')$, via Lemma \ref{4lemRTNS},
    $$ Rep_1 \left( \frac{G}{G'},\frac{G'Y}{G'} \right) = Rep_1 \left( \frac{(G/G')}{(G'Y/G')} \right) = Rep_1(G/G'Y) = |G/G'Y| , $$
    because $G/G'Y$ is abelian.\\
    Now let $n \geq 1$, and let $Y \leqslant H \leqslant G$ be any intermediate subgroup with $|G:H| \leq n$.
    Lemma \ref{4lemRepGSub} (i) applied to $n=1$ provides
     \begin{align*}
|H/H'Y|  & \> = \> 
ab_1(H,Y)  \\ & \> = \> 
 Rep_{_1}(H,Y) \\ & \> \leq \> 
|G:H| \cdot Rep_{_{|G:H|}}(G,Y) \> ,
\end{align*}
     and (ii) follows.
    \end{proof}
    \vspace{0.1cm}
    
    Therefore a bound on $Rep_n(G,Y)$ gives a bound on $ab_n(G,Y)$.
     The converse is not true, highlighting the absence of a generalized version of
      Proposition \ref{propRGBAb} (in the proof, the size of the subgroup $\rho(A_\rho)$,
   which is provided by the Jordan theorem, cannot be bounded by a function on $ab_n(G,Y)$).
      Indeed, a {\itshape robust} notion of BAb does not exist in the generalized setting.
      
   \begin{example} \label{4exDihedral}
   Let $p \geq 3$ be a prime, and let $G_p:=D_{2p}$ be a dihedral group which acts on the $p$-gon $\Omega_p$.
   Let $Y_p \leqslant G_p$ be the stabilizer of a point.
   Since $Y_p$ is maximal and not normal, we have $ab_n(G_p,Y_p)=1$ for all $n \geq 1$.
   On the other hand, all the irreducible representations of a dihedral group have degree $1$ or $2$.
   It follows that $Rep_2(G_p,Y_p) = \frac{p-1}{2}$ can be arbitrarily large.
   \end{example}


	\vspace{1cm}

\end{document}